\documentclass{article}

\newcommand{\self}[1]{}
\newcommand{\selfnote}[1]{}
\newcommand{\selfnoteb}[1]{} 

\usepackage{amsmath,amsfonts,amsthm,units}%
\usepackage{hyperref}                     %
\hypersetup{colorlinks=true}              %

\numberwithin{equation}{section}%

\theoremstyle{plain} \newtheorem{prop}{Proposition}[section]											
\theoremstyle{plain} \newtheorem{lemma}[prop]{Lemma}												
\theoremstyle{plain} 												
\theoremstyle{plain} \newtheorem{defin}[prop]{Definition}												
\theoremstyle{plain} \newtheorem{thm}[prop]{Theorem}												
\theoremstyle{plain} 											
\theoremstyle{plain} 											
\theoremstyle{plain} 											
\theoremstyle{plain} 										
\theoremstyle{plain} 											
\theoremstyle{plain} \newtheorem{claim}[prop]{Claim}												
\theoremstyle{plain} 												
\newsavebox{\detailsbox}													
\newenvironment{details}													
{\begin{lrbox}{\detailsbox}\begin{minipage}{\textwidth}\footnotesize$\ $}				
{\end{minipage}\end{lrbox}\fbox{\usebox{\detailsbox}}}							
{\begin{lrbox}{\detailsbox}\begin{minipage}{\textwidth}}							
{\end{minipage}\end{lrbox}}												

\newcommand{\reals}{{\mathbb R}}																						
\newcommand{\one}{{\mathbf 1}}																						
\newcommand{\One}[2]{\mathbf{1}_{\left[ #1,#2 \right]}}																		
\newcommand{\norm}[1]{\left\Vert#1\right\Vert}																				
\newcommand{\abs}[1]{\left\vert#1\right\vert}																				
\newcommand{\ud}{\,\mathrm{d}}																						
\newcommand{\be}{\begin{equation}}																					
\newcommand{\ee}{\end{equation}}																						
\newcommand{\PR}{\mathbb{P}}																						
\newcommand{\rkhs}{\mathcal{H}}																						
\newcommand{\Ee}{E_{\epsilon}}																						
\newcommand{\Beta}{\mathrm{B}}																						
\newcommand{\DIV}{\operatorname{div}}
\newcommand{\TR}{\operatorname{tr}}

\newcommand{\MID}{\ \ \rule[-0.25cm]{0.1pt}{0.6cm}\ }	

\newlength{\test}																											

 \title{The Onsager-Machlup functional associated with additive fractional noise}%
 \author{Yoha\"{\i} Maayan\footnote{Technion - Israel Institute of Technology. email: \href{mailto:ymaayan@technion.ac.il}{ymaayan@technion.ac.il}.}
 }
 \date{\today}

\begin{document}

\maketitle

\setlength{\parskip}{6pt}
\setlength{\parindent}{0pt}


\begin{abstract}
\normalsize
\noindent
We consider the solution of a stochastic differential equation with additive multidimensional fractional noise. In the case $\frac14<H<\frac12$, we compute the Onsager-Machlup functional (with respect to the driving fractional Brownian motion) for the supremum norm and the H\"{o}lder norms with exponent $\alpha \in \left(0,H-\frac14\right)$ for any element of the Cameron-Martin space $\mathcal H_H$, extending a previous result of Moret and Nualart. In the more general case $H<\frac12$ and $\alpha \in \left(0,H\right)$, we formulate a condition on $h\in\mathcal H_H$ under which the computation of the Onsager-Machlup functional $J\left(h\right)$ follows.
\end{abstract}
\vspace{\stretch{1}} KEY WORDS: Fractional Brownian motion,
Onsager-Machlup functional, Stochastic Integral, Stochastic differential equations, Small ball probabilities.
\newpage



\section{Introduction}

The Onsager-Machlup functional was first introduced by its two physicist eponyms in~\cite{OM_original} and~\cite{OM2_original}. It was later put in precise mathematical context and computed: first by Stratonovich in~\cite{Stratonovich_OM}; then with a slightly different definition by Ikeda and Watanabe in~\cite{Ikeda_Watanabe}, and for diffusions on manifolds by Takahashi and Watanabe in~\cite{Takahashi&Watanabe_OM} and Fujita and Kotani in~\cite{Fujita_Kotani_OM}. We adopt the definition from the latter three papers.

Consider the stochastic `differential' equation with additive noise
\begin{equation}\label{SDE}
X_t=\int_0^tb\left(X_s\right)\ud s+B_t.
\end{equation}
The Onsager-Machlup functional $J\left(\Phi\right)$ is defined by
\begin{equation}\label{OM_def}
e^{J\left(\Phi\right)}=\lim_{\epsilon\searrow 0}\frac{P\left(\norm{X-\Phi}<\epsilon\right)}{P\left(\norm{B}<\epsilon\right)}
\end{equation}
for suitable $\Phi$s.

The norm that was used in the aforementioned papers was the supremum norm $\norm{h}=\sup_{t\in\left[0,1\right]}\abs{h_t}$, and $\left(B_t\right)_{t\in\left[0,1\right]}$ was a Brownian motion. The Onsager-Machlup functional~\eqref{OM_def} was shown to exist for $\Phi\in C^2\left(\left[0,1\right]\right)$ in~\cite{Ikeda_Watanabe}, then extended by Zeitouni in~\cite{Zeitouni_OM} to differentiable $\Phi$s such that $\Phi'$ is H\"{o}lder continuous (of any order), and finally for any $\Phi$ in the Cameron-Martin space associated with the Brownian motion by Shepp and Zeitouni in~\cite{Shepp_Zeitouni}.

The latter two authors then proposed in~\cite{Shepp_Zeitouni_2} a generalisation to other norms $\norm{\cdot}$ in~\eqref{OM_def}. Specifically, the $\alpha$-H\"{o}lder norms for $\alpha<\frac13$ in the multidimensional case, and $\alpha<\frac12$ in the one-dimensional case; and the $L^p$ norms for $p\geq 4$. Capitaine extended in~\cite{Capitaine} the class of norms for which the computation is valid, to include for example: the $\alpha$-H\"{o}lder norm for any $\alpha<\frac12$ in the multidimensional case; and certain fractional Sobolev and Besov norms. Then, in~\cite{Lyons_Zeitouni_OM}, Lyons and Zeitouni suggested a different approach that allows for non-isotropic norms.

The history for the fractional Brownian motion is much shorter. In two distinct one-dimensional cases, D. Nualart and S. Moret obtained in~\cite{Onsager-Machlup} the Onsager-Machlup functional
\begin{equation}\label{OM}
J\left(\Phi\right)=-\frac{1}{2}\left[\norm{\Phi-\int_0^\cdot b\left(\Phi\left(u\right)\right)\ud u}^2_\rkhs+d_H\int_0^1b'\left(\Phi\left(s\right)\right)\ud s\right],
\end{equation}
where $d_H=\left(\frac{2H\Gamma\left(\frac32-H\right)\Gamma\left(H+\frac12\right)}{\Gamma\left(2-2H\right)}\right)^{\frac12}$. In the first case, $\frac14<H<\frac12$ and $\norm{\cdot}$ is either the supremum norm or the $\alpha$-H\"{o}lder norm with $\alpha<H-\frac14$. In this case, they proved~\eqref{OM} for $\Phi$ in the Cameron-Martin space $\rkhs$ associated with the fractional Brownian motion (see Subsection~\ref{subsection:the_CM_space_OM} for details on the Cameron-Martin space), provided that $\Phi$ is H\"{o}lder continuous of an order strictly larger than $\frac12$ (their argument seems to cover any such $\Phi$, though they argued for a slightly smaller class). In the second case, $H>\frac12$ and $\norm{\cdot}$ is the $\alpha$-H\"{o}lder norm with $H-\frac12<\alpha<H-\frac14$. In this case, they proved~\eqref{OM} for any $\Phi$ in the Cameron-Martin space. To the best of our understanding, the value of $d_H$ in~\eqref{OM} should be $1$, the confusion arising from differing normalisations in the literature for the covariance function~\eqref{eq:fBm_covariance} below, see for example~\cite{Decreusefond&Ustunel}. 

In this paper, we prove that the Onsager-Machlup functional exists for the multidimensional fractional Brownian motion $\left(B_t\right)_{t\in\left[0,1\right]}$ with Hurst parameter $\frac14<H<\frac12$, the supremum norm or the $\alpha$-H\"{o}lder norm where $\alpha<H-\frac14$, and any $\Phi\in\rkhs$ in~\eqref{SDE}. We also discuss the case of $\alpha$-H\"{o}lder norms for $H-\frac14\leq\alpha<H$, indicating the condition that $\Phi$ needs to satisfy in order for the Onsager-Machlup functional to exist. In fact, our discussion covers the general case $H<\frac12$ with either the supremum norm or any $\alpha$-H\"{o}lder norm with $\alpha<H$. It would be surprising if this condition on $\Phi$ didn't hold for any $\Phi\in\rkhs$, but this remains open for future research.

We follow methods from~\cite{Shepp_Zeitouni_2} and~\cite{Capitaine}, but also make an attempt to be intrinsic (with respect to the abstract Wiener space associated with the fractional Brownian motion) where we can.

The paper is organised as follows. Section~\ref{section:preliminaries_OM} contains some preliminaries; Section~\ref{section:main_OM} contains the main result which is Theorem~\ref{theorem:OM_functional} and the main points of its proof; and Section~\ref{section:proofs_of_lemmas} contains a few longer proofs of Lemmas from Section~\ref{section:main_OM}.


\section{Preliminaries}
\label{section:preliminaries_OM}
\subsection{Basic definitions and notation}

\begin{defin}
For $H\in\left(0,1\right)$,  a one-dimensional fractional Brownian motion with Hurst parameter $H$ is a centered Gaussian process (defined on the time interval $\left[0,1\right]$ for the sake of notational simplicity) with covariance function
\begin{equation}\label{eq:fBm_covariance}
R\left(s,t\right)=\frac12\left(s^{2H}+t^{2H}-\abs{s-t}^{2H}\right).
\end{equation}
A $d$-dimensional fractional Brownian motion is a process $\left(B_t\right)_{t\in\left[0,1\right]}$ such that $\left(B^i_t\right)$, $i=1,\ldots,d$ are independent one-dimensional fractional Brownian motions.
\end{defin}

For $0\leq\alpha<1$ we consider the Banach space $C^{\alpha}\left[a,b\right]$ of $\alpha$-H\"{o}lder continuous functions with the norm
\begin{equation}
\abs{\varphi}_{C^{\alpha}}=\abs{\varphi}_{\infty}+\sup_{a\leq s<t\leq b}\frac{\abs{\varphi_t-\varphi_s}}{\left(t-s\right)^{\alpha}}.
\end{equation}
We identify the case $\alpha=0$ with $C\left[0,1\right]$. We will use the subscript $0$: $C^{\alpha}_0\left[a,b\right]$, $C_0\left[a,b\right]$ etc. to denote the suitable subspace of functions satisfying $\varphi_0=0$.

For $k\in\mathbb Z_+$ we denote by $C^k_b\left(D,\reals^d\right)$ the class of bounded functions $f:D\to\reals^d$ whose derivatives up to order $k$ are bounded and continuous.

A standard application of Kolmogorov's continuity criterion shows that a fractional Brownian motion (has a modification which) is a.s. $\alpha$-H\"{o}lder continuous with any $\alpha<H$. However, by $p$-variation considerations (see~\cite{Nualart} for example), it is not a semimartingale unless $H=\frac12$, in which case it is a classical Brownian motion. Given $0\leq\alpha<H$, we denote by $P\equiv P_{\alpha}$ the probability measure on $\Omega\equiv \Omega_{\alpha}:=C_0^{\alpha}\left[0,1\right]$ under which the canonical process $B_t\left(\omega\right)=\omega_t$ is a $d$-dimensional fractional Brownian motion.


\subsection{Fractional integrals and derivatives}
\label{subsection:the_CM_space_OM}

We will require the following definitions and properties of fractional integrals and derivatives for the description and analysis of the Cameron-Martin space associated with a fractional Brownian motion. A good reference for this subject is~\cite{fractional}.
\begin{defin}
Let $\varphi\in L^1\left[a,b\right]$ and $\alpha>0$. The left-sided Riemann-Liouville fractional integral of order $\alpha$ is the function $I^{\alpha}_{a+}\varphi:\left(a,b\right]\to\reals$ defined by
\begin{equation}\label{eq:definition_of_fractional_integral_OM}
I^{\alpha}_{a+}\varphi\left(t\right)=\frac{1}{\Gamma\left(\alpha\right)}\int_a^t\frac{\varphi_s}{\left(t-s\right)^{1-\alpha}}\ud s.
\end{equation}
If $\lim_{t\to a+}I^{\alpha}_{a+}\varphi\left(t\right)$ exists, it will be denoted by $I^{\alpha}_{a+}\varphi\left(a\right)$.\\
For $\psi:\left[a,b\right]\to\reals$ and $\alpha\in\left(0,1\right)$, the left-handed Riemann-Liouville fractional derivative of order $\alpha$, when it exists, is the function $D^{\alpha}_{a+}\psi:\left(a,b\right]\to\reals$ defined by
\begin{equation}
D^{\alpha}_{a+}\psi\left(t\right)=\frac{1}{\Gamma\left(1-\alpha\right)}\frac{\ud }{\ud t}\int_a^t\frac{\psi_s}{\left(t-s\right)^{\alpha}}\ud s.
\end{equation}
If $\lim_{t\to a+}D^{\alpha}_{a+}\psi\left(t\right)$ exists, it will be denoted by $D^{\alpha}_{a+}\psi\left(a\right)$.
\end{defin}

It should be emphasized that for general $\alpha$-differentiable $\psi$s, $I^{\alpha}_{a+}D^{\alpha}_{a+}\psi\neq\psi$. However,
\begin{equation}\label{eq:inverse_of_fractional_integral_OM}
D^{\alpha}_{a+}I^{\alpha}_{a+}\varphi=\varphi,\ \ \forall\varphi\in L^1\left[a,b\right].
\end{equation}
\begin{equation}\label{eq:non_inverse_of_fractional_derivative_OM}
I^{\alpha}_{a+}D^{\alpha}_{a+}\psi=\psi,\ \ \forall\psi\in I^{\alpha}_{a+}\left(L^1\left[0,1\right]\right).
\end{equation}
We will often use the notation $I^{\alpha}_{a+}\left[\varphi_t\right]$ for $I^{\alpha}_{a+}\varphi\left(t\right)$.\\
The following boundedness property (from~\cite{fractional} as well) will be of use:
\begin{thm}\label{theorem:fractional_integral_of_weighted_holder_OM}
Let $\beta\in\left(0,1\right)$ such that $0<\alpha+\beta<1$ and let $u<\beta+1$.\\
If $f:\left[0,1\right]\to\reals$ is such that $t^uf\left(t\right)$ is $\beta$-H\"{o}lder continuous and vanishes at $t=0$ then $t^uI_{0+}^{\alpha}f$ is $\left(\beta+\alpha\right)$-H\"{o}lder continuous (and vanishes at $t=0$). Moreover, there is a constant $C>0$ such that
\begin{equation*}
\abs{t^uI_{0+}^{\alpha}f}_{C^{\beta+\alpha}}\leq C\abs{t^uf\left(t\right)}_{C^{\beta}}.
\end{equation*}
If, in addition, $u\geq 0$, then any $g$ such that $t^ug_t$ is $\left(\beta+\alpha\right)$-H\"{o}lder continuous and vanishes at $t=0$ is in $I_{0+}^{\alpha}\left(L^1\left[0,1\right]\right)$; $t^uD_{0+}^{\alpha}g$ is $\alpha$-H\"{o}lder continuous vanishing at $t=0$; and there is some $c>0$ such that 
\begin{equation} \label{eq:fractional_integral_of_weighted_holder_OM}
\abs{t^uD_{0+}^{\alpha}g}_{C^{\alpha}}\leq c\abs{t^ug\left(t\right)}_{C^{\beta+\alpha}}.
\end{equation}
\end{thm}

\subsection{The Cameron-Martin space}

It has been shown in~\cite{Decreusefond&Ustunel} that the Cameron-Martin space associated with a one-dimensional fractional Brownian motion is $\mathcal H=I^{H+\frac12}_{0+}\left(L^2\left[0,1\right]\right)$. Note that $\mathcal H\subset C_0^{\alpha}\left[0,1\right]$ for all $0\leq\alpha<H$. See also~\cite{Bogachev} and~\cite{Janson} for a review and some background on the Cameron-Martin space.\\
The Cameron-Martin space associated with the $d$-dimensional fractional Brownian motion $\left(B_t\right)$ is $\mathcal H^d$ (we will often omit the exponent $d$ in the sequel). 

Decreusefond and \"{U}st\"{u}nel constructed an isometry $K:L^2\left[0,1\right]\to\mathcal H$ in~\cite{Decreusefond&Ustunel}. It has two useful representations.\\
For $H<\frac12$, let
\begin{equation}\label{eq:hypergeometric_function_OM}
F\left(a,b;c;x\right)=\frac{\Gamma\left(c\right)}{\Gamma\left(b\right)\Gamma\left(c-b\right)}\int_0^1t^{b-1}\left(1-t\right)^{c-b-1}\left(1-tx\right)^{-a}\ud t,\ \ \ c>b>0;\ x<1.
\end{equation}
This is the Gauss hypergeometric function (often denoted ${}_2F_1$ in the literature). Now denote
\begin{equation}\label{eq:kernel_of_isometry_OM}
K\left(t,u\right)=c_H\left(t-u\right)^{H-\frac12}F\left(H-\frac12,\frac12-H;H+\frac12;1-\frac{t}{u}\right)\One{0}{t}\left(u\right)
\end{equation}
and
\begin{equation}\label{eq:OM_K_integral_representation}
Kf\left(t\right)=\int_0^1K\left(t,u\right)f_u\ud u,\ \ \ f\in L^2\left[0,1\right].
\end{equation}
Here $c_H=\sqrt{\frac{2H}{\left(1-2H\right)\Beta\left(1-2H,H+\frac12\right)}}$ (note that $K$ differs by a constant from the operator introduced in~\cite{Decreusefond&Ustunel}, since their covariance function $R\left(s,t\right)$ is multiplied by a constant).
\\
Another representation of $K$ is
\begin{equation}
Kf=c_H\Gamma\left(H+\frac12\right)I_{0+}^{2H}t^{\frac12-H}I_{0+}^{\frac12-H}t^{H-\frac12}f,
\end{equation}
which is particularly useful because it is easy to invert:
\begin{equation}
K^{-1}h=\frac{1}{c_H\Gamma\left(H+\frac12\right)}t^{\frac12-H}D_{0+}^{\frac12-H}t^{H-\frac12}D_{0+}^{2H}h,\ \ \ h\in\mathcal H.
\end{equation}
We also note the following simpler formula for $K^{-1}h$ which holds by virtue of~\cite[Formula~(10.6)]{fractional} whenever $h\in C_0^1\left[0,1\right]$:
\begin{equation}\label{eq:Kinverse_simple_representation_OM}
K^{-1}h=\frac{1}{c_H\Gamma\left(H+\frac12\right)}s^{H-\frac12}I_{0+}^{\frac12-H}s^{\frac12-H}h'.
\end{equation}

The isometry $K$ provides a formula for the $\mathcal H$-norm, but it can be quite cumbersome. The following bound is sufficient for many purposes:
\begin{claim}
Assume that $H<\frac12$. If $h$ is differentiable in $\left[0,1\right]$, $h_0=0$ and $h'$ is bounded, then $h\in\mathcal H$ and 
\begin{equation}\label{eq:H_norm_for_C1}
\abs{h}_\rkhs\leq C\abs{h'}_{\infty}.
\end{equation}
\end{claim}
\begin{proof}
The assumptions imply that $h\in C^{H+\frac{1}{2}+\epsilon}\left[0,1\right]$ and $\abs{h}_{C^{H+\frac12+\epsilon}}\leq2\abs{h'}_{\infty}$
for any $\epsilon>0$ such that $H+\frac{1}{2}+\epsilon<1$. We proceed with one such fixed $\epsilon$.

Since $K$ is an isometry,
\begin{equation*}
\abs{h}_{\rkhs}=\abs{t^{\frac12-H}D^{\frac12-H}F}_{L^2},
\end{equation*}
where $F=t^{H-\frac12}D^{2H}h$. Therefore
\begin{equation*}
\abs{h}_{\rkhs}\leq
C_1\abs{t^{\frac12-H}D^{\frac12-H}F}_{C^{\epsilon}}.
\end{equation*}
We use the convention that the H\"older norm of a
non-H\"older function (of the corresponding exponent) is
$\infty$. According to~\eqref{eq:fractional_integral_of_weighted_holder_OM},
\begin{equation*}
\abs{t^{\frac12-H}D^{\frac12-H}F}_{C^{\epsilon}}\leq
C_2\abs{t^{\frac{1}{2}-H}F}_{C^{\epsilon+\frac{1}{2}-H}}.
\end{equation*}
Since
$t^{\frac{1}{2}-H}F=D^{2H}h$, by~\eqref{eq:fractional_integral_of_weighted_holder_OM} again:
\begin{equation*}
\abs{D^{2H}h}_{C^{\epsilon+\frac{1}{2}-H}}\leq
C_3\abs{h}_{C^{\epsilon+\frac{1}{2}+H}}.
\end{equation*}
Thus $\abs{h}_{\rkhs}\leq 2C_3\abs{h'}_{\infty}$.
\end{proof}

In the following theorem and later on as well, we will abuse notation and also use $K$ to denote the isometry $K^d:L^2\left(\left[0,1\right],\reals^d\right)\to\rkhs^d$.
\begin{thm}\label{theorem:relation_between_delta_and_Ito_integral_OM}
Let $\left(B_t\right)_{t\in\left[0,1\right]}$ be a $d$-dimensional fractional Brownian motion. Then there is a $d$-dimensional Brownian motion $\left(W_t\right)_{t\in\left[0,1\right]}$ defined on the same probability space such that
\begin{equation}\label{eq:representation_of_fBm_in_interval_OM}
B_t=\int_0^tK\left(t,s\right)\ud W_s.
\end{equation}
Moreover, if $\alpha\in L^2\left(\left[0,1\right]\times\Omega,\reals^d\right)$ is adapted, then $\delta\left(K\left(u\right)\right)=\int_0^1\alpha_s\ud W_s$ (the latter being the It\^{o} integral).
\end{thm}
Theorem~\ref{theorem:relation_between_delta_and_Ito_integral_OM} was proved in~\cite{Decreusefond&Ustunel}.

\subsection{Malliavin operators and a Girsanov-type theorem}

If $G$ is a Hilbert space and $F\in L^2\left(P;G\right)$ is a random element in $G$, the Malliavin derivative of $F$ (if it exists) is the operator $DF\in L^2\left(P;\mathcal H\otimes G\right)$ which satisfies
\begin{equation}
\lim_{t\to 0}\frac{F\left(\omega+th\right)-F\left(\omega\right)}{t}=DF\left(h\right),\ \ \forall h\in\mathcal H.
\end{equation}
The domain of $D$ is denoted by $\mathbb D^{1,2}\left(G\right)$.\\
The divergence operator $\delta$ is the adjoint of $D$ as an operator from $L^2\left(P;G\right)$ to $L^2\left(P;\mathcal H\otimes G\right)$. Its domain is denoted $\operatorname{dom}\left(\delta\right)$.\\
For a thorough exposition of Malliavin Calculus, see~\cite{Nualart}.

We will need the following multidimensional equivalent of the Girsanov type theorem that was proved in~\cite{Decreusefond&Ustunel} for one-dimensional fractional Brownian motion:
\begin{thm}\label{Girsanov}
Let $u\in \operatorname{dom}\left(\delta\right)\subset L^2\left(P;\mathcal H\right)$ be an adapted process with respect to the filtration generated by $\left(B_t\right)$. Assume that
\begin{equation}\label{exponential_martingale_condition}
Ee^{\delta\left(u\right)-\frac12\abs{u}_{\mathcal H}^2}=1.
\end{equation}
Then $\left(B_t-u_t\right)$ is a $d$-dimensional fractional Brownian motion with Hurst parameter $H$ under the probability measure $Q$ which is defined by
\begin{equation}\label{girsanov_new_measure}
\ud Q=e^{\delta\left(u\right)-\frac12\abs{u}_{\mathcal H}^2}\ud P.
\end{equation}
\end{thm}
This theorem can be proved similarly to the aforementioned Girsanov type Theorem from~\cite{Decreusefond&Ustunel}, and the same remark regarding the Novikov condition also applies: if
\begin{equation}\label{novikov}
Ee^{\frac12\abs{u}_{\mathcal H}^2}<\infty
\end{equation}
then Condition~\eqref{exponential_martingale_condition} follows.


\subsection{Probability asymptotics for small balls}

The Onsager-Machlup functional~\eqref{OM_def} is an asymptotic comparison between $\left(X_t\right)$ and $\left(B_t\right)$ of the probability that a process's path belongs to a small ball with respect to the chosen norm. As such, small ball probability asymptotics for $\left(B_t\right)$ yield corresponding asymptotics for $\left(X_t\right)$. In fact, the former asymptotics will be required in order to compute the Onsager-Machlup functional.

For the fractional Brownian motion $\left(B_t\right)_{t\in\left[0,1\right]}$ with Hurst parameter $H\in\left(0,1\right)$, it was proved in~\cite{small_ball_fBm_supremum} that the small ball probabilities behave as follows for the supremum norm:

\begin{thm} \label{theorem:small_ball_probabilities_supremum_fBm}
	The following limit exists and satisfies:
	\begin{equation} \label{eq:small_ball_supremum_for_fBm}
	\lim_{\epsilon\searrow 0} \epsilon^{\frac{1}{H}}\log P\left(\abs{B}_{\infty}<\epsilon\right)  \in  \left(-\infty,0\right).
	\end{equation}
\end{thm}

For the H\"{o}lder norms, the following estimates were proved in~\cite{small_ball_fBm_holder}:
\begin{thm} \label{theorem:small_ball_probabilities_holder_fBm}
	Let $0<\alpha<H$. Then 
	\begin{equation} \label{eq:small_ball_holder_for_fBm}
	\liminf_{\epsilon\searrow 0}\epsilon^{\frac{1}{H-\alpha}}\log P\left(\abs{B}_{C^{\alpha}}<\epsilon\right),\ \ \ \limsup_{\epsilon\searrow 0}\epsilon^{\frac{1}{H-\alpha}}\log P\left(\abs{B}_{C^{\alpha}}<\epsilon\right)\ \in\ \left(-\infty,0\right).
	\end{equation}
\end{thm}


\subsection{Nuclear operators and approximate limits}

As we shall see in Lemmas~\ref{lemma:Girsanov_argument} and~\ref{exp_terms_lemma}, our computation of the Onsager-Machlup functional will rely on the computation of approximate limits, defined below. These limits are in turn linked with the nuclearity of kernel operators, as was shown by G. Harg\'{e} in~\cite{Harge_second_chaos}. In this subsection, we record the definitions and previously known results that we will need in the sequel; throughout, $\left(W_t\right)$ is a $d$-dimensional \emph{Brownian} motion.

For a finite-dimensional orthogonal projection $Q=\sum_{i=1}^ne_i\otimes e_i$ (where $e_1,\ldots,e_n$ is an orthonormal system) in the Cameron-Martin space $\mathcal H=\int L^2\left[0,1\right]$, define a random element in $\mathcal H$ (which doesn't depend on the particular choice of representation of $Q$ via the $e_i$'s):
\[ \tilde{Q}=\sum_{i=1}^n\delta\left(e_i\right)e_i=\sum_{i=1}^n\left(\int_0^1\dot{e}_i\left(s\right)\cdot\ud W_s\right) e_i. \]
\begin{defin}[by L. Gross, see~\cite{Gross_measurable} \label{definition:measurable_norm} and~\cite{Gross_abstract_Wiener}]
	A norm $\norm{\cdot}$ on $\mathcal H$ is called a \emph{measurable norm} if the following holds: there exists a random variable $N$ such that $N<\infty$ a.s. and for any increasing sequence $\left(Q_n\right)_{n=1}^{\infty}$ of finite-dimensional orthgonal projections in $\mathcal H$ which converges strongly to the identity operator on $\mathcal H$, $\norm{\tilde{Q}_n}\xrightarrow[n\to\infty]{P}N$.
\end{defin}
Consider the setup of Theorem~\ref{theorem:relation_between_delta_and_Ito_integral_OM}. For $d=1$, it was shown in~\cite{Onsager-Machlup} (see Lemma~6) that the norms $\norm{h}=\abs{Kh'}_{\infty}$ and $\norm{h}=\abs{Kh'}_{C^{\alpha}}$ for $\alpha<H$ are measurable norms on $\mathcal H=\int L^2$, with $N=\abs{B}_{\infty}$ and $N=\abs{B}_{C^{\alpha}}$ respectively. The proof for an arbitrary $d$ is identical, and we will not repeat it here.

From now on, as there will be no risk of confusion, we will denote $N\left(\omega\right)$ from Definition~\ref{definition:measurable_norm} by $\norm{\omega}$.

According to~\cite[Theorem~1]{Gross_measurable}, for any measurable norm
, $P\left(\norm{\omega}<\epsilon\right)>0$ for every $\epsilon>0$. This leads us to the following central result:
\begin{thm}[by G. Harg\'{e}, see~\cite{Harge_second_chaos}] \label{theorem:approximate_limit_for_nuclear}
	Let $\norm{\cdot}$ be a measurable norm on $\mathcal H$, and let $k\in L^2\left(\left[0,1\right]^2\right)$ be an a.e. symmetric function. If the operator on $L^2\left[0,1\right]$ defined by 
	 \begin{equation} \label{equation:induced_operator} 
	 f\mapsto \int_0^1k\left(\cdot,t\right)f\left(t\right)\ud t
	 \end{equation}
	 is a nuclear operator, then
	 \begin{equation} \label{eq:approximate_limit_definition}
	 \lim_{\epsilon\searrow 0}E\left( I_2\left(k\right) \MID \norm{\omega}<\epsilon \right)=e^{-\operatorname{Tr}\left(k\right)}
	 \end{equation}
	 where $I_2$ is the double Wiener-It\^{o} integral and $\operatorname{Tr}\left(k\right)$ is the trace of the nuclear operator defined above.
\end{thm}
We will loosely refer to limits of the form~\eqref{eq:approximate_limit_definition} as \emph{approximate limits}; see~\cite{Harge_second_chaos} for precise definitions.

Finally, we state the following result by M. Birman and M. Solomyak which is a special case of~\cite[Theorem~4.1]{nuclear_paper}, tailored for what we will need concerning nuclear operators. Let
\begin{equation}\label{eq:fractional_sobolev_norm}
\abs{h}_{W^{\alpha,2}}=\left(\iint_{\left[0,1\right]^2}\frac{\abs{h_t-h_s}^2}{\abs{t-s}^{\beta}}\ud s\ud t\right)^{\frac12}
\end{equation}
denote the fractional Sobolev norm (see also~\cite{adams} for the definition and some properties of the fractional Sobolev space $W^{\alpha,2}$).
\begin{thm} \label{theorem:nuclear_sufficient_conditions}
	Let $k\left(s,t\right)$ be a measurable function defined for $0\leq s,t\leq 1$ that satisfies
	\begin{equation}
	\abs{\rule{0cm}{0.4cm}\abs{k\left(\cdot,t\right)}_{W^{\alpha,2}}}_{L^2\left(\ud t\right)}<\infty
	\end{equation}
	for some $\alpha>\frac12$. Then the induced operator~\eqref{equation:induced_operator} is a nuclear operator.
\end{thm}





\section{The Onsager-Machlup functional}
\label{section:main_OM}

Before we state the main theorem, we introduce classes of functions in $\mathcal H$ for which the Onsager-Machlup functional will be computed. We fix some $H<\frac12$ throughout.\\
For $0\leq\alpha<H$ we denote by $\mathcal H_{\alpha}$ the set of those $h\in\mathcal H$ that satisfy
\begin{equation}\label{eq:approximate_limit_high_orders_OM}
\limsup_{\epsilon\to 0+}E\left(e^{\delta_1\left(\int_0^{\cdot}g\left(h_t\right)B^m_{1,t}\ud t\right)}|\abs{B}_{C^{\alpha}}<\epsilon\right)\leq 1
\end{equation}
for all $g\in C^1\left(\reals^d\right)$ bounded with bounded derivative and $m=1,\ldots,\left[\frac{1}{2H-2\alpha}\right]$.

The condition~\eqref{eq:approximate_limit_high_orders_OM} can be difficult to check. We will prove below (see Lemma~\ref{lemma:approximate_limit_nuclear_OM}) that~\eqref{eq:approximate_limit_high_orders_OM} holds for $m=1$ for any $h\in\mathcal H$. In particular, $\mathcal H_{\alpha}=\mathcal H$ if $\alpha<H-\frac14$. We conjecture that this in fact the case for any $\alpha<H$, but we have not been able to prove it.

We can now state our main result:

\begin{thm}\label{theorem:OM_functional}
Let $\left(B_t\right)_{t\in\left[0,1\right]}$ be a $d$-dimensional fractional Brownian motion with Hurst parameter $H\in\left(\frac14,\frac12\right)$. Suppose that $\left(X_t\right)_{t\in\left[0,1\right]}$ is a solution of SDE~\eqref{SDE} where $b\in C^2_b\left(\reals^d,\reals^d\right)$. Then the Onsager-Machlup functional~\eqref{OM_def} with respect to the supremum norm and the $\alpha$-H\"{o}lder norms where $\alpha\in\left(0,H-\frac14\right)$ exists for any $h\in\mathcal H$, and it is given by
\begin{equation}\label{main_OM_functional}
J\left(h\right)=-\frac12\left(\abs{h-\int_0^{\cdot}b\left(h_s\right)\ud s}^2_{\mathcal H}+\int_0^1\DIV b\left(h_t\right)\ud t\right).
\end{equation}
Moreover, if $H\in\left(0,\frac12\right)$, $\norm{\cdot}$ is the supremum norm or the $\alpha$-H\"{o}lder norm where $\alpha\in\left(0,H\right)$ and $b\in C^{k+1}_b\left(\reals^d,\reals^d\right)$ for $k>\frac{1}{2H-2\alpha}-1$, then the Onsager-Machlup functional exists for all $h\in\mathcal H_{\alpha}$ and it is given by~\eqref{main_OM_functional}.\\
\end{thm}

The rest of this section is dedicated to the proof of Theorem~\ref{theorem:OM_functional}. We defer the proofs of some lemmas and some other details to Section~\ref{section:proofs_of_lemmas}.

From now on, denote $\norm{\cdot}:=\abs{\cdot}_{C^{\alpha}}$ (for the $\alpha$ from Theorem~\ref{theorem:OM_functional}) and $E_{\epsilon}:=E\left(\cdot\big|\norm{B}<\epsilon\right)$.


We begin by taking advantage of the Girsanov theorem to reduce the computation of the Onsager-Machlup functional to the computation of approximate limits. This is usually the opening point in Onsager-Machlup computations (see~\cite{Ikeda_Watanabe} and~\cite{Onsager-Machlup}).
\begin{lemma}\label{lemma:Girsanov_argument}
	Let $\Phi\in\mathcal H$. Denote
	\begin{align}
	&A_1=\frac{1}{2}\norm{\int_0^\cdot b\left(\Phi_t\right)\ud t}^2_\rkhs-\frac{1}{2}\norm{\int_0^\cdot b\left(\Phi_t+B_t\right)\ud t}^2_\rkhs ,\\
	&A_2=-\delta\left(\Phi\right) ,\\
	&A_3=\left<\int_0^\cdot b\left(\Phi_t+B_t\right)\ud t-\int_0^\cdot b\left(\Phi_t\right)\ud t\ \ \text{\LARGE $,$ }\Phi\right>_\rkhs,\\
	&A_4=\delta\left(\int_0^\cdot b\left(\Phi_t+B_t\right)\ud t\right)+\frac{d_H}{2}\int_0^1\nabla\cdot b\left(\Phi_s\right)\ud s .
	\end{align}
	If
	\begin{equation}\label{main_exp_limit}
	\lim_{\epsilon\to 0}E\left(e^{A_1+A_2+A_3+A_4}\mid\norm{B}<\epsilon\right)=1
	\end{equation}
	then the statement in Equation~\eqref{main_OM_functional} follows.
\end{lemma}


The following Lemma allows us to consider each term separately, which further simplifies the approximate limits to be considered. It was first stated and proved in~\cite{Ikeda_Watanabe} (see p. 536).
\begin{lemma}\label{exp_terms_lemma}
	If $\limsup_{\epsilon\to 0}E\left(e^{cA_i}\mid\norm{B}<\epsilon\right)\leq 1$ for all $c\in\reals$ and $i=1,\ldots,n$, then
	\begin{equation*}
	\lim_{\epsilon\to 0}E\left(e^{A_1+\cdots+A_n}\mid\norm{B}<\epsilon\right)=1
	\end{equation*}
\end{lemma}


We will show that $A_1$, $A_2$, $A_3$ and $A_4$ all satisfy the assumption in Lemma~\ref{exp_terms_lemma}.

For $A_1$ and $A_3$, it will suffice to show that
\begin{equation}\label{dominance}
\norm{\int_0^\cdot b\left(\Phi_t+B_t\right)\ud t-\int_0^\cdot b\left(\Phi_t\right)\ud t}_{\rkhs}\leq C\norm{B}
\end{equation}
for some constant $C>0$, since this implies that
\begin{equation*}
\Ee e^{cA_i}\leq e^{cC\epsilon},\ \ i=1,3.
\end{equation*}
\begin{proof}[Proof of~\eqref{dominance}]
We will once again use Inequality~\eqref{eq:H_norm_for_C1}. Setting
\begin{equation*}
\Psi_s=\int_0^s \left[b\left(\Phi_t+B_t\right)-b\left(\Phi_t\right)\right]\ud t,
\end{equation*}
we have
\begin{equation*}
\Psi'_s=b\left(\Phi_s+B_s\right)-b\left(\Phi_s\right)
\end{equation*}
which is bounded. Therefore
\begin{equation}
\norm{\Psi}_\rkhs\leq \norm{b\left(\Phi_s+B_s\right)-b\left(\Phi_s\right)}_\infty\leq \norm{b'}_\infty\norm{B}_\infty\leq \norm{b'}_\infty\norm{B}.
\end{equation}
This completes the proof.
\end{proof}

The term corresponding to $A_2$ will be treated in a manner inspired by~\cite{Shepp_Zeitouni}. That paper dealt with the one dimensional Brownian case; we will prove that
\begin{equation}\label{deterministic_integrand}
\lim_{\epsilon\to 0}E\left(e^{\delta\Psi}\mid\norm{B}<\epsilon\right)=1,\ \ \Psi\in\rkhs
\end{equation}
for the multidimensional fractional Brownian motion. 
\begin{proof}[Proof of~\eqref{deterministic_integrand}]
Under the conditioning $\norm{B}<\epsilon$ we have that
\[ -\epsilon<B_s^1<\epsilon,\ \ 0\leq s\leq 1; \]
thus, for any $c\in\reals$,
\[ e^{-\abs{c}\epsilon}<E\left(e^{cB_s^1}\mid \norm{B}<\epsilon\right)<e^{\abs{c}\epsilon}. \]
This proves the basic property
\begin{equation}\label{norm_exp_property}
\lim_{\epsilon\to 0}E\left(e^{cB_s^1}\mid\norm{B}<\epsilon\right)=1,\ \ \forall c\in\reals,\ 0\leq s\leq 1.
\end{equation}
We now make the same remark that was made in~\cite{Shepp_Zeitouni}: by Jensen's inequality and symmetry,
\begin{equation*}
E\left(e^{\delta\Psi}\mid\norm{B}<\epsilon\right)\geq 1,\ \ \forall\epsilon>0.
\end{equation*}
We therefore want to prove the inverse limit inequality.

If $\Psi=\sum_{i=1}^ka_iR\left(t_i,\cdot\right)$ then $\delta\Psi=\sum_{i=1}^ka_i\cdot B_{t_i}$. Thus from Lemma~\ref{exp_terms_lemma} and from~\eqref{norm_exp_property},
\begin{equation}\label{simple_case_deterministic_integrand_proof}
\lim_{\epsilon\to 0}E\left(e^{\delta\Psi}\mid\norm{B}<\epsilon\right)=1.
\end{equation}
Suppose now that $\Psi\in\rkhs$, and let $\eta>0$. Then there is some $\Psi_0$ which is a finite linear combination as above and such that $\norm{\Psi-\Psi_0}_\rkhs<\eta$. By the Cauchy-Schwartz inequality,
\begin{equation}\label{approximation_deterministic_integrand_proof}
E\left(e^{\delta\Psi}\mid\norm{B}<\epsilon\right)\leq \sqrt{\Ee e^{2\delta\Psi_0}} \sqrt{\Ee e^{2\delta\left(\Psi-\Psi_{0}\right)}}.
\end{equation}
The first term above tends to $1$ as $\epsilon\to 0$ by~\eqref{simple_case_deterministic_integrand_proof}. For the second term, we will prove that
\begin{equation}\label{expectation_bound_H_norm_deterministic_integrand_proof}
\Ee e^{\delta\Phi}\leq e^{\frac{1}{2}\norm{\Phi}^2_\rkhs},\ \ \forall \Phi\in\rkhs.
\end{equation}
Together with~\eqref{approximation_deterministic_integrand_proof}, we will then have
\begin{equation}
\limsup_{\epsilon\to 0}\Ee e^{\delta\Psi}\leq e^{\frac{\eta^2}{2}},\ \ \forall \eta>0.
\end{equation}
This completes the proof.

To prove~\eqref{expectation_bound_H_norm_deterministic_integrand_proof}, fix some $\Phi\in\rkhs$ and let $\left(\Phi_n\right)_{n=1}^{\infty}$ be some orthonormal basis of $\rkhs$ such that $\Phi_1=\frac{\Phi}{\norm{\Phi}_\rkhs}$. We will make use of the (a.s.) representation $B=\sum_{n=1}^{\infty}X_n\Phi_n$, where $X_n:=\delta\Phi_n$, and of a conclusion from~\cite[Theorem~2.1]{GEOPSS}: denote
\begin{equation}
C=\left\{\mathbf{x}\in\reals^\infty \left| \norm{\sum_{n=1}^{\infty}x_n\Phi_n}<\epsilon \right.\right\}.
\end{equation}
The norm $\norm{\cdot}$ is the usual H\"{o}lder norm that we've been using. This defines a convex symmetric set. It was shown in~\cite{Shepp_Zeitouni} that~\cite[Theorem~2.1]{GEOPSS} then implies
\begin{equation}\label{infinite_gaussian_correlation_inequality_deterministic_integrand_proof}
P\left(\abs{X_1}\leq h\right)P\left(\mathbf{X}\in C\right)\leq P\left(\abs{X_1}\leq h,\mathbf{X}\in C\right).
\end{equation}
Indeed, from~\cite[Theorem~2.1]{GEOPSS} we have~\eqref{infinite_gaussian_correlation_inequality_deterministic_integrand_proof} for the projection of $\mathbf{X}$ and $C$ onto $N$-dimensional space; taking $N\to\infty$ by means of the monotone convergence theorem implies~\eqref{infinite_gaussian_correlation_inequality_deterministic_integrand_proof}. We may now deduce the inequality
\begin{equation}
P\left(\abs{\frac{\delta\Phi}{\norm{\Phi}_\rkhs}}\leq h\right)\leq P\left(\left.\abs{\frac{\delta\Phi}{\norm{\Phi}_\rkhs}}\leq h\right|\norm{B}<\epsilon\right),\ \ \ \forall h\geq 0.
\end{equation}
From this we arrive at 
\begin{equation}
P_\epsilon \left(e^{\abs{\delta\Phi}}> h\right)\leq P \left(e^{\abs{\delta\Phi}}> h\right),\ \ \ \forall h\geq 0,
\end{equation}
which readily implies the inequality
\begin{equation}
\Ee e^{\delta\Phi}\leq\Ee e^{\abs{\delta\Phi}}\leq E e^{\abs{\delta\Phi}}=e^{\frac{1}{2}\norm{\Phi}^2_\rkhs}.
\end{equation}
This proves~\eqref{expectation_bound_H_norm_deterministic_integrand_proof} and thus completes the proof.
\end{proof}

The term corresponding to $A_4$ is treated using a Taylor expansion of $b$, which is where we apply methods from~\cite{Capitaine}. The novelty of that paper was that it allowed for a Taylor expansion of an order greater than $2$. In this way we will be able to discuss H\"older norms of any order less than $H$. In addition, we remove the restriction on $\Phi$ from~\cite{Onsager-Machlup}, which was needed for the calculations corresponding to the first order term. Let us write the Taylor expansion of order $k=\left[\frac{1}{2H-2\alpha}\right]$ of $b\left(\Phi_s+B_s\right)$ as
\begin{equation}\label{taylor}
b\left(\Phi_s+B_s\right)=b\left(\Phi_s\right)+b'\left(\Phi_s\right) B_s+\sum_{i=2}^k\frac{1}{m!}b^{\left(m\right)}\left(\Phi_s\right)\left[B_s,\ldots,B_s\right]+R\left(s,B_s\right)
\end{equation}
where the remainder term $R$ actually depends on $b$ and $\Phi$ as well. Note that by the assumptions on $b$, we have the inequality
\begin{equation}\label{remainder_bound}
\abs{R\left(s,x\right)}\leq M\abs{x}^{k+1},\ \ x\in\reals^d.
\end{equation}
The constant $M$ depends only on $b$. In addition, $R\left(s,B_s\right)$ can be seen to be an $\alpha$-H\"{o}lder function for any $\alpha<H$: Each term in the Taylor expansion has this property as products of such (or compositions of $C^1$ functions on such), and the left-hand-side is also such a function as a composition.
Using this expansion, we may write  $A_4$ as the sum of five types of terms:
\begin{align}
&Z_1=\delta\left(\int_0^\cdot b\left(\Phi_t\right)\ud t\right),       \label{Z_1}\\
&Z_2=\delta^i\left(\int_0^\cdot b'_{ii}\left(\Phi_t\right)B^i_t\ud t\right)+\frac{1}{2}\int_0^1b'_{ii}\left(\Phi_s\right)\ud s,       \label{Z_2}\\
&Z_3=\delta^i\left(\int_0^\cdot b'_{ij}\left(\Phi_t\right)B^j_t\ud t\right),\  \ \ \ \ \ \ \  i\neq j,       \label{Z_3}\\
&Z_4=\delta\left(\int_0^\cdot \frac{1}{m!}b^{\left(m\right)}\left(\Phi_t\right)\left[B_t,\ldots,B_t\right]\ud t\right),\ \ \ \ \ \ \ \ 2\leq m\leq k,       \label{Z_4}\\
&Z_5=\delta\left(\int_0^\cdot R\left(t,B_t\right)\ud t\right).       \label{Z_5}
\end{align}
We will deal with each one of these terms separately in accordance with Lemma~\ref{exp_terms_lemma}.

Starting with $Z_1$, note that $\int_0^sb\left(\Phi_t\right)\ud t$ is a continuously differentiable function of $s$ on $\left[0,1\right]$. By Inequality~\eqref{eq:H_norm_for_C1}, it belongs to~$\rkhs$. Therefore by~\eqref{deterministic_integrand}, 
\begin{equation}
\lim_{\epsilon\to 0}E\left(e^{cZ_1}\mid\norm{B}<\epsilon\right)=1
\end{equation}
for all $c\in\reals$.


For $Z_2$,~\eqref{Z_2} follows from the following Lemmas.

\begin{lemma}\label{lemma:differentiability_and_trace_of_processes}
	Let $\left(B_t\right)_{t\in\left[0,1\right]}$ be a one-dimensional fractional Brownian motion with Hurst parameter $H<\frac12$, and let $p:\left[0,1\right]\to\reals$ be a continuous function.
\\
	Consider the process $u_t=\int_0^tp_sB_s^m\ud s$. Then:
	\begin{enumerate}
		\item
		For any $m\in\mathbb N$, the process $\left(u_t\right)$
		belongs to $\mathbb D^{1,2}\left(\mathcal H\right)$, and
		\begin{equation}\label{eq:stochastic_gradient_of_process_OM}
		Du\left(h\right)\left(t\right)=\int_0^tmp_sB_s^{m-1}h_{s}\ud s,\ \ h\in\mathcal H.
		\end{equation}
		Furthermore,
		\begin{equation}\label{eq:HS_norm_OM}
		\abs{Du}_{\operatorname{HS}}=m^2B_H^2\int_0^1\int_0^tt^{2H-1}\left(\int_u^t\frac{s^{\frac12-H}K\left(s,u\right)}{\left(t-s\right)^{H+\frac12}}p_sB_s^{m-1}\ud s\right)^2\ud u\ud t
		\end{equation}
		where $B_H=\frac{1}{c_H\Gamma\left(H+\frac12\right)\Gamma\left(\frac12-H\right)}$ and $K\left(s,u\right)$ was defined in Equation~\eqref{eq:kernel_of_isometry_OM}.
		
		\item
		If $m=1$ and $p_s=G\left(h_s\right)$ where $h\in\mathcal H^n$ and $G$ is a Lipschitz continuous function, then $\operatorname{Sym}\left(Du\right)=\frac12\left(Du+Du^{\ast}\right)$ is a trace-class operator on $\mathcal H$ and
		\begin{equation}\label{eq:trace_of_Du_OM}
		\TR\left(\operatorname{Sym}\left(Du\right)\right)=\frac{1}{2}\int_0^1G\left(h_s\right)\ud s.
		\end{equation}
	\end{enumerate}
\end{lemma}


\begin{lemma}\label{lemma:approximate_limit_nuclear_OM}
	Let $\left(B_t\right)$ be a $d$-dimensional fractional Brownian motion with Hurst parameter $H<\frac12$, $G$ a Lipschitz continuous function, $0\leq\alpha<H$ and $h\in\mathcal H$. Then:
	\begin{equation}
	\lim_{\epsilon\to 0+}E_{\epsilon}\left(e^{\delta_i\left(\int_0^{\cdot}G\left(h_t\right)B_{i,t}\ud t\right)}\right)  =  e^{-\frac{1}{2}\int_0^1G\left(h_s\right)\ud s}.
	\end{equation}
\end{lemma}


\begin{proof}
	Denote $\beta_t=K^{-1}\left[\int_0^tG\left(h_s\right)B_{s,i}\ud s\right]$, where $K$ is the isometry reviewed in Subsection~\ref{subsection:the_CM_space_OM}. According to Lemma~\ref{lemma:differentiability_and_trace_of_processes} and a standard transfer principle, $\beta\in\mathbb D_W^{1,2}\left(L^2\left[0,1\right]\right)$ and 
	\begin{equation}
	D_W\beta=K^{-1}D\left(\int_0^{\cdot}G\left(h_s\right)B_{s,i}\ud s\right)K.
	\end{equation}
	In particular, it follows from~\eqref{eq:stochastic_gradient_of_process_OM} that $D_W\beta$ is deterministic. Let $k\left(t,u\right)$ denote its Hilbert-Schmidt kernel. We have computed $k\left(t,u\right)$ in the proof of Lemma~\ref{lemma:differentiability_and_trace_of_processes}: see~\eqref{eq:kernel_of_unitary_equivalence_of_DU_OM}. However, we do not require the expression here. All we need is that $\beta_t=\int_0^1k\left(t,u\right)\ud W_u$. 
	Therefore
	\begin{equation}
	\delta\left(\int_0^{\cdot}G\left(h_s\right)B_{s,i}\ud s\right)=\int_0^1\beta_t\ud W_t=\int_0^1\int_0^1k\left(t,u\right)\ud W_u\ud W_t
	\end{equation}
	belongs to the second chaos. Denote the double Wiener integral (with respect to $\left(W_t\right)$) by $I_2$ and let $\hat{k}\left(t,u\right)$ be the symmetrization of $k\left(t,u\right)$. By Lemma~\ref{lemma:differentiability_and_trace_of_processes}, $\hat{k}\left(t,u\right)$ is the kernel of a nuclear operator on $L^2\left[0,1\right]$. Thus according to Theorem~\ref{theorem:approximate_limit_for_nuclear}, $e^{I_2\left(\hat{k}\right)}$ admits an approximate limit with respect to the measurable seminorm $\norm{B}$:
	\begin{equation}
	\limsup_{\epsilon\to0+}E\left(e^{I_2\left(\hat{k}\right)}\mid\norm{B}<\epsilon\right)=e^{-\TR\left(\hat{k}\right)}=e^{-\frac{1}{2}\int_0^1G\left(h_s\right)\ud s}.
	\end{equation}
	This completes the proof.
\end{proof}

We note that Lemma~\ref{lemma:approximate_limit_nuclear_OM} also shows that~\eqref{eq:approximate_limit_high_orders_OM} holds for any $h\in\mathcal H$ if $m=1$.


We now move on to $Z_3$. Since $i\neq j$, $B^j$ is independent of the fractional Brownian motion beneath the divergence operator $\delta^i$. Furthermore, since $\int_0^\cdot b'_{ij}\left(\Phi_t\right)B^j_t\ud t\in\rkhs$ almost surely and
\begin{equation}
\norm{\int_0^\cdot b'_{ij}\left(\Phi_t\right)B^j_t\ud t}_\rkhs\leq\norm{b'_{ij}\circ\Phi\cdot B^j}_{\infty},
\end{equation}
it follows that
\begin{equation}
\limsup_{\epsilon\to 0}\Ee e^{c\norm{\int_0^\cdot b'_{ij}\left(\Phi_t\right)B^j_t\ud t}^2_\rkhs}\leq 1,\ \ \forall c\in\reals.
\end{equation}
The following Lemma completes the treatment of $Z_3$. Denote 
\begin{equation*} 
\hat{\mathcal F}^i=\sigma\left(B^j_t,\ 0\leq t\leq 1,\ j\neq i\right).
\end{equation*}
\begin{lemma}\label{exponential_limit_independent}
Let $\Psi\in\rkhs$ almost surely be such that $\Psi^i$ is $\hat{\mathcal F}^i$-measurable for $i=1,\ldots,d$ and that for any $c\in\reals$, 
\begin{equation}
\limsup_{\epsilon\to 0}\Ee e^{c\norm{\Psi}_\rkhs^2}\leq 1.
\end{equation}
 Then
\begin{equation}
\lim_{\epsilon\to 0}E\left(e^{\delta\Psi}\mid\norm{B}<\epsilon\right)=1.
\end{equation}
\end{lemma}
Lemma~\ref{exponential_limit_independent} and its proof are inspired from~\cite[Theorem~1]{Shepp_Zeitouni_2}. See also~\cite{Shepp_Zeitouni}.
\begin{proof}
Since $\delta\Psi=\delta_1\Psi^1+\cdots+\delta_d\Psi^d$, it suffices by Lemma~\ref{exp_terms_lemma} to show that 
\begin{equation}
\forall c\in\reals\ \ \limsup_{\epsilon\to 0}\Ee e^{c\delta_i\Psi^i}\leq 1,\ \ \ i=1,\ldots,d.
\end{equation}
Let's take $i=1$. We will use the notation: 
\begin{equation*}
\mathbb EX=E\left(X\mid \hat{\mathcal F}^1\right),
\end{equation*}
We then have the following typical equality:
\begin{equation}\label{first_estimate_exp_limit_independent_proof}
\Ee e^{c\delta_1\Psi^1}=\Ee\left(\mathbb E_{\epsilon}e^{c\delta_1\Psi^1}\right),
\end{equation}
where $\mathbb E_{\epsilon}=\mathbb E\left(\cdot\mid \norm{B}<\epsilon\right)$.
We will also make use of the regular conditional probability 
\begin{equation}
\mathbb P\left(A\right)=P\left(A\mid \hat{\mathcal F}^1\right),\ \ A\in\mathcal F;
\end{equation} 
details can be found, for example, in ~\cite{Dudley} (see~Theorem~10.2.2 in particular). This is a random probability measure on $\left(\Omega,\mathcal F\right)$. Fix some $\omega\in\Omega=C\left(\left[0,1\right],\reals^d\right)$ for which $\mathbb P_{\omega}$ is well defined as a probability measure (these $\omega$s have full probability). We will need the fact that for almost all such $\omega$,
\begin{equation}\label{components_are_deterministic_under_regular_conditional_probability}
B^i=\omega^i\ \ \PR_{\omega}\text{-a.s.},\ \ i=2,\ldots,d.
\end{equation}
The proof of Equation~\eqref{components_are_deterministic_under_regular_conditional_probability} is most naturally written in the language of conditional distributions: consider for the next few lines 
\begin{equation*}
C\left(\left[0,1\right];\reals^{d-1}\times\reals^d\right)=C\left(\left[0,1\right];\reals^{d-1}\right)\times C\left(\left[0,1\right];\reals^d\right),
\end{equation*}
using the notation $\left(x,\omega\right)$ for its elements. We endow this space with its Borel $\sigma$-algebra and the image probability measure $Q=P\circ \iota^{-1}$, where 
\begin{equation*}
\begin{cases}
\iota:C\left(\left[0,1\right];\reals^d\right)\to C\left(\left[0,1\right];\reals^{d-1}\times\reals^d\right)\\
\iota\left(\omega\right)=\left(\rule{0cm}{0.4cm}\left(\omega^2,\ldots,\omega^d\right),\omega\right).
\end{cases}
\end{equation*}
In the language of \cite[Theorem~10.2.1]{Dudley}, where $\omega$ replaces $y$, the conditional distributions $Q_{x}$ on $C\left(\left[0,1\right];\reals^d\right)$ exist and
\begin{equation*}
1=Q\left(\left\{\omega^i=x^{i-1},\ i=2,\ldots,d\right\}\right)=E\left(Q_x\left(\left\{\omega^i=B^{i},\ i=2,\ldots,d\right\}\right)\right).
\end{equation*}
 Since $Q_{\hat{\omega}_1}=\PR_{\omega}$ for $P$-almost all $\omega$ by uniqueness, we obtain~\eqref{components_are_deterministic_under_regular_conditional_probability}.

 The internal conditioning, on the right hand side of~\eqref{first_estimate_exp_limit_independent_proof}, on the event $\norm{B}<\epsilon$, should now be thought off as being on $B^1$; as such, notice that it is symmetric and convex. It follows, by the same argument as in the proof of~\eqref{deterministic_integrand}, that
\begin{equation}
\mathbb E\left(\left. e^{c\delta_1\Psi^1}\right| \norm{B}<\epsilon\right)\leq \mathbb E e^{\abs{c\delta_1\Psi^1}},\ \ P\text{-a.s.};
\end{equation}
since, under $\PR$, $\delta_1\left(c\Psi^1\right)\sim N\left(0,c^2\norm{\Psi^1}^2_\rkhs \right)$,
\begin{equation}\label{last_estimate_exp_limit_independent_proof}
\mathbb E e^{\abs{c\delta_1\Psi^1}}=e^{\frac{c^2}{2}\norm{\Psi^1}^2_\rkhs }.
\end{equation}
Estimates~\eqref{first_estimate_exp_limit_independent_proof}-\eqref{last_estimate_exp_limit_independent_proof} can now be summarised:
\begin{equation}
\Ee e^{c\delta_1\Psi^1}\leq \Ee e^{\frac{c^2}{2}\norm{\Psi^1}^2_\rkhs }.
\end{equation}
By the assumption, it follows that
\begin{equation}
\limsup_{\epsilon\to 0}\Ee e^{c\delta_1\Psi^1}\leq 1,\ \ \forall c\in\reals.
\end{equation}
This completes the proof. 
\end{proof}

We will treat $Z_4$ in a manner inspired by~\cite{Capitaine} (See~\cite[Lemma~4]{Capitaine}). Fix an $m$ ($2\leq m\leq k$); we can write $Z_4$ as the sum of terms of the form
\begin{equation}\label{z4_terms}
\delta_i\left(\int_0^\cdot g\left(\Phi_t\right)\left(B_{i_1}^{r_1}B_{i_2}^{r_2}\cdots B_{i_q}^{r_q}\right)_t\ud t\right),
\end{equation}
where $1\leq i,i_1,\ldots,i_q\leq d$, $i_1,\ldots,i_q$ are all distinct, $r_1,\ldots,r_q\in\mathbb{N}$ satisfy $r_1+\cdots+r_q=m$ and $g\in C^1$ is a bounded function with bounded derivative. We will show that each one of those terms, $A$, satisfies 
\begin{equation}\label{z4_terms_exponential_limit}
\limsup_{\epsilon\to 0}E_{\epsilon}e^{cA}\leq 1.
\end{equation}
 This will complete the treatment of $Z_4$ by Lemma~\ref{exp_terms_lemma}. We will henceforth assume without loss of generality that $i=1$ in~\eqref{z4_terms}.

If $i_1,\ldots,i_q\neq 1$ then~\eqref{z4_terms_exponential_limit} follows from Lemma~\ref{exponential_limit_independent}. For the case $i_1=1$, we have the following (notice the slight relabelling of indices):
\begin{prop}
Let $g\in C^1\left(\reals^d;\reals\right)$ be a bounded function with bounded derivative and let $h\in\mathcal H_{\alpha}$. Fix $2\leq m\leq k$. Then for any $1<i_1<i_2<\cdots<i_q\leq d$ and $r,r_1,\ldots,r_q\in\mathbb{N}$ such that $r+r_1+\cdots+r_q=m$,
\begin{equation}
\lim_{\epsilon\to 0}E_{\epsilon}e^{\delta_1\left[\int_0^\cdot g\left(h_t\right)\left(B_1^rB_{i_1}^{r_1}B_{i_2}^{r_2}\cdots B_{i_q}^{r_q}\right)_t\ud t\right]}=1.
\end{equation}
\end{prop}
\begin{proof}
We will prove this by induction on $q=0,\ldots,m\wedge d-1$.

For $q=0$ we have $r=m\geq 2$ and therefore we need to prove that
\begin{equation*}
\lim_{\epsilon\to 0}E_{\epsilon}e^{\delta_1\left[\int_0^\cdot g\left(h_t\right)B^m_{1,t}\ud t\right]}=1.
\end{equation*}
This is true since $h\in\mathcal H_{\alpha}$ and $m<\left[\frac{1}{2H-2\alpha}\right]$.

Assume that the proposition holds for some $0\leq q\leq m\wedge d-2$; we now prove that it holds for $q+1$. More specifically, we will actually show that
\begin{equation}\label{z4_induction_step}
\limsup_{\epsilon\to 0}E_{\epsilon}e^{\delta_1\left[\int_0^\cdot g\left(h_t\right)\left(B_{1,t}^rB_{i_1,t}^{r_1}\cdots B_{i_{q+1},t}^{r_{q+1}}\right)\ud t\right]}\leq 1.
\end{equation}
Combined with Lemma~\ref{exp_terms_lemma}, and since $cg$ is also $C^1$, bounded, with bounded derivative for any $c\in\reals$, this will complete the proof. By polarising the following monomial:
\begin{equation*}
a_1a_2\cdots a_{r+r_1}=\frac{1}{\left(r+r_1\right)!2^{r+r_1}}\sum_{\mathbf{\epsilon}\in\left\{\pm 1\right\}^{r+r_1}}\epsilon_1\cdots\epsilon_{r+r_1}\left(\epsilon_1a_1+\cdots+\epsilon_{r+r_1}a_{r+r_1}\right)^{r+r_1},
\end{equation*}
we get the identity
\begin{align*}
a^rb^{r_1}&=\frac{1}{\left(r+r_1\right)!2^{r+r_1}}\sum_{\mathbf{\epsilon}\in\left\{\pm 1\right\}^{r+r_1}}\epsilon_1\cdots\epsilon_{r+r_1}\cdot\\
&\cdot\left[\left(\epsilon_1+\cdots+\epsilon_r\right)a+\cdots+\left(\epsilon_{r+1}+\cdots+\epsilon_{r+r_1}\right)b\right]^{r+r_1}.
\end{align*}
If $\epsilon_1+\cdots+\epsilon_r=\epsilon_{r+1}+\cdots+\epsilon_{r+r_1}=0$ then there's no contribution from that $\mathbf{\epsilon}$. Otherwise, we can find some $0\leq\theta<2\pi$ (depending on that $\mathbf{\epsilon}$) such that
\begin{equation*}
\cos\theta=\frac{\epsilon_1+\cdots+\epsilon_r}{\sqrt{\left(\epsilon_1+\cdots+\epsilon_r\right)^2+\left(\epsilon_{r+1}+\cdots+\epsilon_{r+r_1}\right)^2}}
\end{equation*}
and
\begin{equation*}
-\sin\theta=\frac{\epsilon_{r+1}+\cdots+\epsilon_{r+r_1}}{\sqrt{\left(\epsilon_1+\cdots+\epsilon_r\right)^2+\left(\epsilon_{r+1}+\cdots+\epsilon_{r+r_1}\right)^2}}.
\end{equation*}
In summary, if $l$ is the number of such $\mathbf{\epsilon}$s/$\theta$s, then for appropriate constants $c_1,\ldots,c_l$, we have:
\begin{equation}\label{final_polarization_formula}
a^rb^{r_1}=\sum_{i=1}^lc_i\left(a\cos\theta_i-b\sin\theta_i\right)^{r+r_1}.
\end{equation}
If we apply~\eqref{final_polarization_formula} to the exponent in~\eqref{z4_induction_step} with $a=B_1$ and $b=B_{i_1}$, we get
\begin{equation*}
\delta_1\left[\int_0^\cdot g\left(h_t\right)\left(B_1^rB_{i_1}^{r_1}\cdots B_{i_{q+1}}^{r_{q+1}}\right)_t\ud t\right]=\sum_{i=1}^lc_iI_i,
\end{equation*}
where
\begin{equation*}
I_i=\delta_1\left[\int_0^\cdot g\left(h_t\right)\left(B_1\cos\theta_i-B_{i_1}\sin\theta_i\right)^{r+r_1}_t\left(B_{i_2}^{r_2}\cdots B_{i_{q+1}}^{r_{q+1}}\right)_t\ud t\right].
\end{equation*}
We will show that $\limsup_{\epsilon\to 0}E_{\epsilon}e^{I_i}\leq 1$ for each fixed $i$, as per Lemma~\ref{exp_terms_lemma} again. Let $R$ be the $2\times 2$ matrix representing rotation of angle $\theta_i$:
\begin{equation*}
R=
\begin{pmatrix}
\cos\theta_i & -\sin\theta_i\\
\sin\theta_i  & \cos\theta_i
\end{pmatrix},
\end{equation*}
and let
\begin{equation*}
A=
\begin{pmatrix}
R_{2\times 2} 		             &       0_{2\times\left(d-2\right)}\\
0_{\left(d-2\right)\times 2}     &       \operatorname{Id}_{\left(d-2\right)\times\left(d-2\right)}
\end{pmatrix}.
\end{equation*}
The matrix $A$ is orthogonal and therefore $W_t:=AB_t$ is a $d$-dimensional fractional Brownian motion with the same Hurst parameter. In addition, $\norm{W}=\norm{B}$, so that $E_{\epsilon}$ is the same for both processes. Note the following relations ($\tilde{\delta}$ denotes the divergence operators associated with $W$):
\begin{align*}
&W_1=B_1\cos\theta_i-B_{i_1}\sin\theta_i,\\
&B_1=W_1\cos\theta_i+W_2\sin\theta_i,\\
&\delta_1=\cos\theta_i\tilde{\delta}_1+\sin\theta_i\tilde{\delta}_2.
\end{align*}
These imply
\begin{align*}
I_i&=\cos\theta_i\tilde{\delta}_1\left[\int_0^\cdot g\left(h_t\right)\left(W_1^{r+r_1}W_{i_2}^{r_2}\cdots W_{i_{q+1}}^{r_{q+1}}\right)_t\ud t\right]\\
&\ \ \ \ \ +\sin\theta_i\tilde{\delta}_2\left[\int_0^\cdot g\left(h_t\right)\left(W_1^{r+r_1}W_{i_2}^{r_2}\cdots W_{i_{q+1}}^{r_{q+1}}\right)_t\ud t\right].
\end{align*}
We have
\begin{equation*}
\limsup_{\epsilon\to 0}E_{\epsilon}e^{\tilde{\delta}_1\left[\int_0^\cdot g\left(h_t\right)\left(W_1^{r+r_1}W_{i_2}^{r_2}\cdots W_{i_{q+1}}^{r_{q+1}}\right)_t\ud t\right]}\leq 1
\end{equation*}
by the induction hypothesis and the spherical symmetry of the norm, and, since $i_2>2$,
\begin{equation*}
\limsup_{\epsilon\to 0}E_{\epsilon}e^{\tilde{\delta}_2\left[\int_0^\cdot g\left(h_t\right)\left(W_1^{r+r_1}W_{i_2}^{r_2}\cdots W_{i_{q+1}}^{r_{q+1}}\right)_t\ud t\right]}\leq 1
\end{equation*}
by Lemma~\ref{exponential_limit_independent}. Lemma~\ref{exp_terms_lemma} now completes the proof.
\end{proof}
For $Z_5$, note that for any $\delta_0>0$,
\begin{align}\label{eq:Z5_inequality_OM}
E\left(e^{cZ_5}\mid\norm{B}<\epsilon\right)&=\int_0^\infty P\left(e^{cZ_5}>y\mid\norm{B}<\epsilon\right)\ud y\\
=\int_0^{e^{\delta_0}}&P\left(e^{cZ_5}>y\mid\norm{B}<\epsilon\right)\ud y+\int_{e^{\delta_0}}^\infty P\left(e^{cZ_5}>y\mid\norm{B}<\epsilon\right)\ud y\nonumber\\
&\leq e^{\delta_0}+\int_{\delta_0}^\infty P\left(cZ_5>x\mid\norm{B}<\epsilon\right)e^x\ud x\nonumber
\end{align}
By Theorem~\ref{theorem:relation_between_delta_and_Ito_integral_OM}, 
\begin{equation}
c\delta\left(\int_0^\cdot R\left(t,B_t\right)\ud t\right)=c\int_0^1K^{-1}\left[\int_0^sR\left(t,B_t\right)\ud t\right]\cdot\ud W_s
\end{equation}
where $\left(W_t\right)_{t\in\left[0,1\right]}$ is the $d$-dimensional Brownian motion in Equation~\eqref{eq:representation_of_fBm_in_interval_OM}. Note that $K^{-1}\left[\int_0^sR\left(t,B_t\right)\ud t\right]$ is adapted with respect to the filtration generated by $\left(W_t\right)$.\\ 
Let 
\begin{equation}
M_t=c\int_0^tK^{-1}\left[\int_0^sR\left(u,B_u\right)\ud u\right]\cdot\ud W_s.
\end{equation}
This is a one-dimensional martingale which satisfies
\begin{align}\label{eq:cross_variation_OM}
\left<M\right>_t&=c^2\int_0^t\abs{K^{-1}\left[\int_0^sR\left(u,B_u\right)\ud u\right]}^2\ud s\\
&\leq c^2\int_0^1\abs{K^{-1}\left[\int_0^sR\left(u,B_u\right)\ud u\right]}^2\ud s=c^2\abs{\int_0^sR\left(u,B_u\right)\ud u}_{\mathcal H}^2.\nonumber
\end{align}
Combining~\eqref{eq:H_norm_for_C1} and~\eqref{remainder_bound}, we get
\begin{equation}\label{eq:H_norm_Z5_OM}
\norm{c\int_0^\cdot R\left(t,B_t\right)\ud t}_{\rkhs}\leq \tilde{M}\epsilon^{k+1},\ \ \ P_\epsilon\text{-a.s.}
\end{equation}
Therefore by~\eqref{eq:cross_variation_OM} and~\eqref{eq:H_norm_Z5_OM},
\begin{equation}\label{eq:cross_variation2_OM}
\left<M\right>_t\leq \tilde{M}\epsilon^{2k+2},\ \ \ \forall t\in\left[0,1\right];\ P_\epsilon\text{-a.s.}
\end{equation}
By a classical Theorem, 
there exists some Brownian motion (possibly defined on an extension of the probability space; we keep denoting the probability measure by $P$, as there is no risk of error) $\left(Y_t\right)_{t\geq 0}$ such that $M_t=Y_{\left<M\right>_t}$. It follows that
\begin{align*}
P\left(\left.c\delta\left(\int_0^\cdot R\left(t,B_t\right)\ud t\right)>x\right|\norm{B}<\epsilon\right)&=\frac{P\left(M_1>x,\norm{B}<\epsilon\right)}{P\left(\norm{B}<\epsilon\right)}\\
&\hspace{-3.5cm}=\frac{P\left(Y_{\left<M\right>_1}>x,\norm{B}<\epsilon\right)}{P\left(\norm{B}<\epsilon\right)}\leq \frac{P\left(\max_{0\leq t\leq \left<M\right>_1}\abs{Y_t}>x,\norm{B}<\epsilon\right)}{P\left(\norm{B}<\epsilon\right)}\\
\text{\scriptsize(by~\eqref{eq:cross_variation2_OM})}&\leq \frac{P\left(\max_{0\leq t\leq \tilde{M}\epsilon^{2k+2}}\abs{Y_t}>x,\norm{B}<\epsilon\right)}{P\left(\norm{B}<\epsilon\right)}.
\end{align*}
Since $P\left(\max_{0\leq t\leq u}\abs{Y_t}>x\right)\leq\sqrt{\frac{u}{2\pi}}\frac{4}{x}e^{-\frac{x^2}{2u}}$ 
and $P\left(\norm{B}<\epsilon\right)\geq e^{-\frac{c}{\epsilon^{\frac{1}{H-\alpha}}}}$, 
we have
\begin{equation}
P\left(cZ_5>x\mid\norm{B}<\epsilon\right)\leq \frac{\sqrt{\frac{\tilde{M}}{2\pi}}\frac{4}{x}\epsilon^{k+1}e^{-\frac{x^2}{2\tilde{M}\epsilon^{2k+2}}}}{e^{-\frac{c}{\epsilon^{\frac{1}{H-\alpha}}}}}=\frac{C\epsilon^{k+1}}{x}e^{\frac{c}{\epsilon^{\frac{1}{H-\alpha}}}-\frac{x^2}{2\tilde{M}\epsilon^{2k+2}}}.
\end{equation}
Going back to~\eqref{eq:Z5_inequality_OM},
\begin{equation}
E\left(e^{cZ_5}\mid\norm{B}<\epsilon\right)\leq e^{\delta_0}+\frac{C\epsilon^{k+1}}{\delta_0}\int_{\delta_0}^{\infty}e^{\frac{c}{\epsilon^{\frac{1}{H-\alpha}}}-\frac{x^2}{2\tilde{M}\epsilon^{2k+2}}}e^x\ud x.
\end{equation}
If $2k+2>\frac{1}{H-\alpha}$, which is the condition on $k$ in Theorem~\ref{theorem:OM_functional}, then the right-hand-side converges to $e^{\delta_0}$ as $\epsilon\to 0+$. 
Letting $\delta_0\to 0+$, we obtain 
\begin{equation}
\limsup_{\epsilon\to 0+}E\left(e^{cZ_5}\mid\norm{B}<\epsilon\right)\leq 1.
\end{equation}
\medskip
This completes the proof of Theorem~\ref{theorem:OM_functional}.


\section{Proofs of the Lemmas}
\label{section:proofs_of_lemmas}

\begin{proof}[Proof of Lemma~\ref{lemma:differentiability_and_trace_of_processes}]
For $h\in\mathcal H$ and $r\in\reals$, thinking of $\left(B_t\right)$ as the canonical process,
\begin{equation}
u_t\left(\omega+rh\right)=\int_0^tp_s\left(B_s+rh_s\right)^m\ud s.
\end{equation}
Therefore
\begin{equation}
u_t\left(\omega+rh\right)-u\left(\omega\right)=mr\int_0^t p_sB_s^{m-1}h_s\ud s+\int_0^t\sum_{i=2}^m\binom{m}{i}p_sr^ih_s^iB_s^{m-i}\ud s.
\end{equation}
Rearranging,
\begin{equation}\label{eq:computation_of_H_derivative}
\frac{u_t\left(\omega+rh\right)-u\left(\omega\right)}{r}-m\int_0^t p_sB_s^{m-1}h_s\ud s=\int_0^t\sum_{i=2}^m\binom{m}{i}p_sr^{i-1}h_s^iB_s^{m-i}\ud s.
\end{equation}
The right-hand-side of~\eqref{eq:computation_of_H_derivative} is a.s. $C^1$, so according to Inequality~\eqref{eq:H_norm_for_C1}
\begin{equation}
\abs{\left(\int_0^t\sum_{i=2}^m\binom{m}{i}p_sr^{i-1}h_s^iB_s^{m-i}\ud s\right)_{t\in\left[0,1\right]}}_{\mathcal H}\leq \max_{s\in\left[0,1\right]}\abs{\sum_{i=2}^m\binom{m}{i}p_sr^{i-1}h_s^iB_s^{m-i}}\xrightarrow[r\to 0]{\text{a.s}.}0.
\end{equation}
This will prove~\eqref{eq:stochastic_gradient_of_process_OM} once its right-hand-side is shown to be a Hilbert-Schmidt operator on $\mathcal H$. Note that by Inequality~\eqref{eq:H_norm_for_C1}, the right-hand-side of~\eqref{eq:stochastic_gradient_of_process_OM} belongs to $\mathcal H$ for any $h\in\mathcal H$; so it is indeed an operator on $\mathcal H$. We denote it temporarily by $T$:

\begin{equation}\label{eq:trace_proof_Du}
Th\left(t\right)=\int_0^tg_sh_s\ud s\ \ \ \ \ \text{{\scriptsize ($g_s:=mp_sB_s^{m-1}$)}}.
\end{equation}

Recall the isometry $K$ between $L^2\left[0,1\right]$ and $\mathcal H$ which was reviewed in Subsection~\ref{subsection:the_CM_space_OM}. The operator $\overline{T}:=K^{-1}\circ T\circ K$ on $L^2\left[0,1\right]$ is unitarily equivalent to $T$. According to Equations~\eqref{eq:trace_proof_Du} and~\eqref{eq:OM_K_integral_representation},
\begin{equation}
T\left(Kf\right)\left(t\right)=\int_0^tg_s\int_0^1K\left(s,u\right)f_u\ud u\ud s.
\end{equation}
Thus
\begin{equation}
D_{0+}^{2H}\circ T\circ K\left(f\right)\left(t\right)=\frac{1}{\Gamma\left(1-2H\right)}\frac{\ud }{\ud t}\int_0^t\frac{1}{\left(t-r\right)^{2H}}\int_0^rg_s\int_0^1K\left(s,u\right)f_u\ud u\ud s\ud r,
\end{equation}
and so
\begin{align*}
\overline{T}f\left(t\right)&=\frac{t^{\frac12-H}}{c_H\Gamma\left(H+\frac12\right)\Gamma\left(\frac12-H\right)\Gamma\left(1-2H\right)}\cdot\\
&\hspace{0.8cm}\cdot\frac{\ud }{\ud t}\int_0^t\frac{v^{H-\frac12}}{\left(t-v\right)^{\frac12-H}}\frac{\ud }{\ud v}\int_0^v\frac{1}{\left(v-r\right)^{2H}}\int_0^rg_s\int_0^1K\left(s,u\right)f_u\ud u\ud s\ud r\ud v\\
&=\int_0^1k\left(t,u\right)f_u\ud u,
\end{align*}
where
\begin{align*}
k\left(t,u\right)&=A_Ht^{\frac12-H}\frac{\partial }{\partial t}\int_0^t\frac{v^{H-\frac12}}{\left(t-v\right)^{\frac12-H}}\frac{\partial }{\partial v}\int_0^v\frac{1}{\left(v-r\right)^{2H}}\int_0^rg_sK\left(s,u\right)\ud s\ud r\ud v\\
&=A_Ht^{\frac12-H}\int_0^1K\left(s,u\right)g_s\frac{\partial }{\partial t}\int_0^t\frac{v^{H-\frac12}}{\left(t-v\right)^{\frac12-H}}\frac{\partial }{\partial v}\int_0^v\frac{1}{\left(v-r\right)^{2H}}\One{s}{1}\left(r\right)\ud r\ud v\ud s\\
&=\Gamma\left(2H\right)\Gamma\left(\frac12-H\right)A_Ht^{\frac12-H}\int_0^1K\left(s,u\right)g_sD_{0+;t}^{\frac12-H}\left[t^{H-\frac12}D_{0+;t}^{2H}\left[\One{s}{1}\left(t\right)\right]\right]\ud s
\end{align*}
($A_H=\frac{1}{c_H\Gamma\left(H+\frac12\right)\Gamma\left(\frac12-H\right)\Gamma\left(1-2H\right)}$). By $D_{0+;t}^{2H}$ we mean $D_{0+}^{2H}$ acting on its argument as a function of $t$, with all other variables frozen (and mutatis mutandis for similar expressions). Note that
\begin{equation}
D_{0+;t}^{2H}\left[\One{s}{1}\left(t\right)\right]=\frac{1}{\Gamma\left(2H\right)}\frac{\partial }{\partial t}\int_0^t\frac{1}{\left(t-r\right)^{2H}}\One{s}{1}\left(r\right)\ud r=\frac{\One{s}{1}\left(t\right)}{\Gamma\left(2H\right)\left(t-s\right)^{2H}},
\end{equation}
whence
\begin{equation}
D_{0+;t}^{\frac12-H}\left[t^{H-\frac12}D_{0+;t}^{2H}\left[\One{s}{1}\left(t\right)\right]\right]=D_{0+;t}^{\frac12-H}\left[\frac{\One{s}{1}\left(t\right)t^{H-\frac12}}{\Gamma\left(2H\right)\left(t-s\right)^{2H}}\right].
\end{equation}
According to the definition of the fractional derivative, and~\cite[Equation~(2.49), p. 41]{fractional},
\begin{align*}
D_{0+;t}^{\frac12-H}\left[\frac{\One{s}{1}\left(t\right)t^{H-\frac12}}{\Gamma\left(2H\right)\left(t-s\right)^{2H}}\right]&=D_{s+;t}^{\frac12-H}\left[\frac{t^{H-\frac12}}{\Gamma\left(2H\right)\left(t-s\right)^{2H}}\right]\One{s}{1}\left(t\right)\\
&=\frac{\Gamma\left(1-2H\right)s^{\frac12-H}\One{s}{1}\left(t\right)}{\Gamma\left(2H\right)\Gamma\left(\frac12-H\right)t^{1-2H}\left(t-s\right)^{H+\frac12}}.
\end{align*}
Therefore
\begin{equation}
k\left(t,u\right)=A_H\Gamma\left(1-2H\right)t^{H-\frac12}\int_0^1\frac{s^{\frac12-H}\One{s}{1}\left(t\right)}{\left(t-s\right)^{H+\frac12}}K\left(s,u\right)g_s\ud s
\end{equation}
Taking into account the structure of the kernel $K\left(s,u\right)$ in~\eqref{eq:kernel_of_isometry_OM}, we arrive at
\begin{equation}\label{eq:kernel_of_unitary_equivalence_of_DU_OM}
k\left(t,u\right)=B_Ht^{H-\frac12}\int_u^t\frac{s^{\frac12-H}}{\left(t-s\right)^{H+\frac12}}K\left(s,u\right)g_s\ud s,\ \ \ u<t
\end{equation}
($B_H=A_H\Gamma\left(1-2H\right)$; and $k\left(t,u\right)=0$ for $t<u$).

The following estimate is~\cite[Theorem~3.2]{Decreusefond&Ustunel}:
\begin{equation}
\abs{K\left(s,u\right)}\leq C_Hu^{H-\frac12}\left(s-u\right)^{H-\frac12},\ \ \ u<s.
\end{equation}
Thus
\begin{equation}
\abs{k\left(t,u\right)}\leq D_Hu^{H-\frac12}\max_{s\in\left[0,1\right]}\abs{g_s}\int_u^t\frac{1}{\left(t-s\right)^{H+\frac12}\left(s-u\right)^{\frac12-H}}\ud s,\ \ \ u<t.
\end{equation}
The change of variables $x=\frac{s-u}{t-u}$ yields
\begin{equation}\label{eq:Beta_miracle_OM}
\int_u^t\left(s-u\right)^{H-\frac12}\left(t-s\right)^{-H-\frac12}\ud s=\int_0^1x^{H-\frac12}\left(1-x\right)^{-H-\frac12}\ud x=\Beta\left(\frac12+H,\frac12-H\right),
\end{equation}
and thus we arrive at
\begin{equation}\label{eq:estimate_on_kernel_of_Du_OM}
\abs{k\left(t,u\right)}\leq E_Hu^{H-\frac12}\max_{s\in\left[0,1\right]}\abs{g_s},\ \ \ u<t.
\end{equation}
Since the right-hand-side of~\eqref{eq:estimate_on_kernel_of_Du_OM} is square integrable, it follows that $\tilde{T}$ (and therefore $T$) is a Hilbert-Schmidt operator, which completes the proof of~\eqref{eq:stochastic_gradient_of_process_OM} and of~\eqref{eq:HS_norm_OM} (the latter by~\eqref{eq:kernel_of_unitary_equivalence_of_DU_OM}). Moreover,
\begin{equation}
\abs{Du}_{\operatorname{HS}}\leq F_H\max_{s\in\left[0,1\right]}\abs{g_s}\leq G_H\abs{p}_{\infty}\max_{s\in\left[0,1\right]}\abs{B_s^{m-1}}.
\end{equation}
On the other hand, according to Inequality~\eqref{eq:H_norm_for_C1},
\begin{equation*}
\abs{u}_\rkhs=\abs{\int_0^\cdot p_sB_s^m\ud s}_\rkhs\leq \abs{p}_{\infty}\max_{s\in\left[0,1\right]}\abs{B_s^m}.
\end{equation*}
Since $E\left(\max_{s\in\left[0,1\right]}\abs{B_s^{k}}\right)<\infty$ for any $k$,
we obtain $E\left(\abs{u}_{\mathcal H}^2\right),E\left(\abs{Du}_{\operatorname{HS}}^2\right)<\infty$. Thus $u\in\mathbb D^{1,2}\left(\mathcal H\right)$.\medskip

From this point and onward, assume that $m=1$ and $p_s=G\left(h_s\right)$ where $h\in\mathcal H$ and $G$ is Lipschitz continuous, as in Part 2 of the Lemma. We now proceed with the proof that the symmetrization $\operatorname{Sym}\left(Du\right)$ is a trace-class operator.

The kernel operator defined by $k\left(t,u\right)$ in~\eqref{eq:kernel_of_unitary_equivalence_of_DU_OM} was considered in~\cite{Onsager-Machlup}. In particular, it was shown (see Lemma~13 and its proof) that for a suitable constant $c$, the symmetrization of the kernel $k\left(t,u\right)-cG\left(h_t\right)\One{0}{t}\left(u\right)$ defines a trace-class operator on $L^2\left[0,1\right]$. It will therefore suffice to show that the symmetrization of the function $G\left(h_t\right)\One{0}{t}\left(u\right)$, which is $G\left(h_{t\vee u}\right)$, defines a trace-class operator. We will show that it satisfies the assumptions of Theorem~\ref{theorem:nuclear_sufficient_conditions} with $\alpha=H+\frac12$, i.e.:
\begin{equation}
\abs{\abs{G\left(h_{t\vee \cdot}\right)}_{W^{\frac12+H,2}}}_{L^2\left(\ud t\right)}<\infty
\end{equation}
We have already noted in Subsection~\ref{subsection:the_CM_space_OM} that $\mathcal H=I_{0+}^{\frac12+H}\left(L^2\left[0,1\right]\right)$. According to Theorem~18.3 and Remark~18.1 in~\cite{fractional}, $I_{0+}^{\frac12+H}\left(L^2\left[0,1\right]\right)\subset W^{\frac12+H,2}$. Thus
\begin{equation}\label{eq:sobolev_finite_integral_OM}
\iint_{\left[0,1\right]^2}\frac{\abs{h_t-h_s}^2}{\abs{t-s}^{\beta}}\ud t\ud s<\infty
\end{equation}
where $\beta=2+2\left(\frac12+H\right)$. Similarly,
\begin{equation} \label{equation:sobolev_G_norm}
\abs{G\left(h_{t\vee \cdot}\right)}_{W^{\frac12+H,2}}^2=\iint_{\left[0,1\right]^2}\frac{\abs{G\left(h_{t\vee v}\right)-G\left(h_{t\vee u}\right)}^2}{\abs{v-u}^{\beta}}\ud u\ud v
\end{equation}
We can split the square into four parts: $Q:=\left[0,1\right]^2=A_t\cup B_t\cup C_t\cup D_t$, where
\[ \begin{cases}
A_t=\left\{ \left(u,v\right)\in Q\mid u\leq t\leq v \right\}\\
B_t=\left\{ \left(u,v\right)\in Q\mid t\leq u,v \right\}\\
C_t=\left\{ \left(u,v\right)\in Q\mid v\leq t\leq u \right\}\\
D_t=\left\{ \left(u,v\right)\in Q\mid u,v\leq t \right\}
\end{cases} \]
Clearly, 
\begin{equation} \label{equation:sobolev_D} \iint_{D_t}\frac{\abs{G\left(h_{t\vee v}\right)-G\left(h_{t\vee u}\right)}^2}{\abs{v-u}^{\beta}}\ud u\ud v=0
\end{equation}
and
\begin{equation} \label{equation:sobolev_B}
\iint_{B_t}\frac{\abs{G\left(h_{t\vee v}\right)-G\left(h_{t\vee u}\right)}^2}{\abs{v-u}^{\beta}}\ud u\ud v\leq M\iint_{B_t}\frac{\abs{h_v-h_u}^2}{\abs{v-u}^{\beta}}\ud u\ud v
\end{equation}
where $M$ is the Lipschitz constant of $G$. Now, on $A_t$, $\abs{v-u}\geq \abs{v-t}$ and therefore
\begin{equation} \label{equation:sobolev_A}
\iint_{A_t}\frac{\abs{G\left(h_{t\vee v}\right)-G\left(h_{t\vee u}\right)}^2}{\abs{v-u}^{\beta}}\ud u\ud v\leq M\iint_{A_t}\frac{\abs{h_v-h_t}^2}{\abs{v-t}^{\beta}}\ud u\ud v\leq M\int_{0}^1\frac{\abs{h_v-h_t}^2}{\abs{v-t}^{\beta}}\ud v.
\end{equation}
In the same way we also have
\begin{equation} \label{equation:sobolev_C}
\iint_{C_t}\frac{\abs{G\left(h_{t\vee v}\right)-G\left(h_{t\vee u}\right)}^2}{\abs{v-u}^{\beta}}\ud u\ud v \leq M\int_{0}^1\frac{\abs{h_t-h_u}^2}{\abs{t-u}^{\beta}}\ud u.
\end{equation}
Putting~\eqref{equation:sobolev_G_norm}-\eqref{equation:sobolev_C} together, we obtain:
\begin{equation}
\abs{G\left(h_{t\vee \cdot}\right)}_{W^{\frac12+H,2}}^2\leq M\abs{h}_{W^{\frac12+H,2}}^2+2M\int_{0}^1\frac{\abs{h_v-h_t}^2}{\abs{v-t}^{\beta}}\ud v,
\end{equation}
so that
\begin{equation}
\abs{\abs{G\left(h_{t\vee \cdot}\right)}_{W^{\frac12+H,2}}}_{L^2\left(\ud t\right)}^2\leq 3M\abs{h}_{W^{\frac12+H,2}}^2<\infty.
\end{equation}
Therefore according to Theorem~\ref{theorem:nuclear_sufficient_conditions}, $G\left(h_{t\vee u}\right)$ defines a trace-class operator.
 \medskip

It remains to prove~\eqref{eq:trace_of_Du_OM}.\\
According to~\cite[Theorem~3.1]{Trace_of_kernel_operators}, 
\begin{equation}\label{eq:trace_general_OM}
\TR\left(\overline{T}\right)=\int_0^1\left(\lim_{r\to 0}\frac{1}{4r^2}\iint_{C_r\left(v,v\right)}k\left(t,u\right)\ud t\ud u\right)\ud v
\end{equation}
where $C_r\left(v,v\right)=\left[v-r,v+r\right]^2$.
Fix some $v>0$. Plugging~\eqref{eq:kernel_of_isometry_OM} into~\eqref{eq:kernel_of_unitary_equivalence_of_DU_OM}, we get (for $t>u$)
\begin{equation}\label{eq:kernel_2_of_unitary_equivalence_of_DU_OM}
k\left(t,u\right)=c_HB_Ht^{H-\frac12}\int_u^t\frac{s^{\frac12-H}F\left(H-\frac12,\frac12-H;H+\frac12;1-\frac{s}{u}\right)g_s}{\left(s-u\right)^{\frac12-H}\left(t-s\right)^{H+\frac12}}\ud s.
\end{equation}
Since the hypergeometric function defined in~\eqref{eq:hypergeometric_function_OM} is continuous in $\left(-1,1\right)$ (see~\cite{transcendental_functions} for details), the numerator in the integrand of~\eqref{eq:kernel_2_of_unitary_equivalence_of_DU_OM} is continuous at $\left(s,u\right)=\left(v,v\right)$. On the other hand, by~\eqref{eq:Beta_miracle_OM} and since $c_HB_H\Beta\left(\frac12+H,\frac12-H\right)=1$, we can denote:
\begin{equation}
G\left(t,u,s\right)=t^{H-\frac12}s^{\frac12-H}F\left(H-\frac12,\frac12-H;H+\frac12;1-\frac{s}{u}\right)g_s,
\end{equation}
and get the following inequality for $v>r>0$:
\begin{equation}
\min_{v-r\leq t,u,s\leq v+r }\left\{G\left(t,u,s\right)\right\}\leq k\left(t,u\right)\leq \max_{v-r\leq t,u,s\leq v+r }\left\{G\left(t,u,s\right)\right\}.
\end{equation}
Therefore by continuity,
 noting that $F\left(a,b;c;0\right)=1$, we conclude that 
\begin{equation}
\lim_{r\to 0}\frac{1}{4r^2}\iint_{C_r\left(v,v\right)}k\left(t,u\right)\ud t\ud u=\frac12g_v,\ \ v>0.
\end{equation}
Thus according to Equation~\eqref{eq:trace_general_OM}
\begin{equation}
\TR\left(T\right)=\TR\left(\overline{T}\right)=\frac12\int_0^1g_v\ud v.
\end{equation}
\end{proof}

\begin{proof}[Proof of Lemma~\ref{lemma:Girsanov_argument}]
Recall SDE~\eqref{SDE}:
\begin{equation}\label{SIE}
X_t=\int_0^tb\left(X_s\right)\ud s+B_t.
\end{equation}
Set 
\begin{equation}\label{eq:Girsanov_shift}
u_t=\int_0^tb\left(\Phi_s+B_s\right)\ud s-\Phi_t,\ \ 0\leq t\leq 1.
\end{equation}
The Novikov condition~\eqref{novikov} follows from the fact that $\abs{u}_{\mathcal H}$ is a.s. bounded.\\
Since $u$ is adapted, we may apply Theorem~\ref{Girsanov}: the process $B_t-u_t$ is a fractional Brownian motion (with the same Hurst parameter $H$) with respect to the probability measure
\begin{equation*}
\ud Q=e^{\delta u-\frac{1}{2}\norm{u}_\rkhs^2}\ud P.
\end{equation*}
We have thus built a `weak solution' $\left(Q,B_t-u_t,B_t+\Phi_t\right)$ of SDE~\eqref{SDE}. That is, $B_t+\Phi_t$ solves~\eqref{SIE} if $B_t$ is replaced with $B_t-u_t$. It follows that
\begin{equation}
P\left(\norm{X-\Phi}<\epsilon\right)=Q\left(\norm{B}<\epsilon\right)=E\left(\one_{\left\{\norm{B}<\epsilon\right\}}e^{\delta u-\frac{1}{2}\norm{u}_\rkhs^2}\right).
\end{equation}
Going back to~\eqref{main_OM_functional}, we have come to
\begin{equation}
e^{J\left(\Phi\right)}=\lim_{\epsilon\to 0}E\left(e^{\delta u-\frac{1}{2}\norm{u}_\rkhs^2}\mid \norm{B}<\epsilon\right).
\end{equation}
Taking into account the definition of $u$ in~\eqref{eq:Girsanov_shift} and the proposed Onsager-Machup functional in Equation~\eqref{main_OM_functional}, the Lemma is proved.
\end{proof}

\bibliography{bib_OM}
\bibliographystyle{plain}

\end{document}